\documentclass[11pt,reqno,oneside]{amsart}
\usepackage[english]{babel}
\usepackage[utf8]{inputenc}
\usepackage[mathscr]{eucal}
\usepackage{color}
\usepackage{xcolor}
\usepackage{amsfonts}   
\usepackage{amssymb}
\usepackage{latexsym}
\usepackage{enumerate}
\usepackage{mathtools}

\usepackage[colorlinks,citecolor=purple,linkcolor=blue]{hyperref}

\usepackage{geometry}

\let\nc\newcommand

\nc{\la}{\label}

\newtheorem{theorem}{Theorem}[section]

\newtheorem{proposition}[theorem]{Proposition}
\theoremstyle{definition}
\newtheorem{definition}[theorem]{Definition}

\theoremstyle{remark}

\newcommand{\into}{\,\,\hookrightarrow\,\,}

\makeatletter

\newcommand{\colim@}[2]{%
  \vtop{\m@th\ialign{##\cr
    \hfil$#1\operator@font colim$\hfil\cr
    \noalign{\nointerlineskip\kern1.5\ex@}#2\cr
    \noalign{\nointerlineskip\kern-\ex@}\cr}}%
}
\newcommand{\colim}{%
  \mathop{\mathpalette\colim@{\rightarrowfill@\scriptscriptstyle}}\nmlimits@
}
\renewcommand{\varprojlim}{%
  \mathop{\mathpalette\varlim@{\leftarrowfill@\scriptscriptstyle}}\nmlimits@
}
\renewcommand{\varinjlim}{%
  \mathop{\mathpalette\varlim@{\rightarrowfill@\scriptscriptstyle}}\nmlimits@
}
\makeatother

\newenvironment{dedication}
  {
   \itshape             
   \raggedleft          
  }

\title{A note on the defintion of derived functors}

\author{Jo\~ao Schwarz}
\address{Shenzhen International Center for Mathematics, SUSTech, 1088 Xueyuan Avenue, Shenzhen 518055,
P.R. China}
\email{jfschwarz.0791@gmail.com}

\begin{document}

\medskip

\medskip

\begin{abstract}
    The purpose of this note is to consider in detail the construction of derived functors. The classical construction, such as in Cartan-Eilenberg or Grothendieck, is clarified, and it is shown, at the same time, that everything can be formalized in \textbf{ZFC}, unlike the approach using derived categories. Our work is done in a more general context in which the codomain of our functors is any Grothendieck category, not necessarily abelian groups.
\end{abstract}

\maketitle

\begin{dedication}
    Dedicated to Gabi.
\end{dedication}

\medskip

\medskip

\section{Introduction}

In many popular books about homological algebra (e.g. \cite{Weibel}, \cite{Rotman}, \cite{HS}), or in the classical works of Cartan and Eilenberg \cite{CE} and Grothendieck's famous Tôhoku paper \cite{Grothendieck}, to obtain, say, the left derived funtors $L_n(F)$ of a right exact functor $F$ from an abelian category $\mathcal{C}$ with enough projectives to the category of abelian groups $\mathfrak{Ab}$, the authors obtain a projective resolution of an object $C \in \mathcal{C}$, denoted here by $(\mathsf{P}, \epsilon)$ (that is, $\mathsf{P}$ is a chain complex of projective objects and the chain complex $\mathsf{P} \xrightarrow{\epsilon} C$ is exact), and then define, for each $n \in \mathbb
{N}$, the n-th left derived functor $L_n(F)$ as the n-th cohomology group $H^n(F(\mathsf{P}))$, by noticing that, had we choosen another projective resolution $\mathsf{Q}$, there exists a canonical isomorphism between $H^n(F(\mathsf{P}))$ and $H^n(F(\mathsf{Q}))$ (see, for instance, \cite[Proposition 5.1]{HS}).

However, to have a functor we \emph{need} to fix once and for all an object of $\mathfrak{Ab}$, and the set of isomorphisms classes of $H^n(F(\mathsf{P}))$, first, is actually not set (it is in bijection with the universe $\mathsf{V}$), and second, it obviously doesn't have an abelian group structure. This small gap seems to missing from the original sources as well the textbooks mentioned above. The first purpose of this note is to a give a full and detailed explanation of the missing last step of the classical approach - forthe best of our knowledge, is also the first writing account of it.

It is important to mention that the modern approach, using derived categories, successfully solves the issues raised above and has many advantages with respect to the theory of $\delta$-functors in \cite{Grothendieck}. For an excellent introduction to derived categories and, as an example of its success, the theory of $D$-modules, see, for instance \cite{Borel}. However, regarding its set-theoretical foundations, a series of important technical issues appear. To the best of our understanding, there is no clear exposition of the foundations of derived categories in the literature. As remarked in \cite[10.3.3, 10.3.6]{Weibel}, the set theoretical foundations are not considered in the classical references \cite{Hartshorne}, \cite{Verdier}, \cite{GZ}, and even if we follow the popular approach of Universes, the following can happen: fixing an Universe $\mathcal{U}$ and a locally small category $\mathcal{C}$ in $\mathcal{U}$, the localization of the category may not be again locally small.

Given this, we stress the second aim of this note: to show that the classical theory of $\delta$-functors can be completely formalized in \textbf{ZFC}; the proof of which involves some simple but not trivial arguments within set theory and category theory, to be discussed bellow. That the theory of $\delta$-functors in \cite{Grothendieck} could be carried out in \textbf{ZFC} was probably expected, but again no detailed exposition of this - to the best of this author knowledge - is to be found in the literature. We avoid the use of the axiom of choice as much as possible: not because of any objections against it, of course, but to make our construction the more ``natural" possible. Also, at the end of this paper, we stress some other foundational issues associated with choosing more than just \textbf{ZFC} for a formalization of the theory of derived functors, arising from logic (cf. \cite{Kunen}, \cite{Foundations}).

The third clarification is which kind of abelian categories are allowed as a  target category of our derived functors in order for the theory to be formalized in \textbf{ZFC}. We show that, instead of restricting ourselves to $\mathfrak{Ab}$, we can work over any Grothendieck category, and it is an open problem how much the theory can be extended to more general abelian categories and the theory of derived functors still be defined in \textbf{ZFC}.

The mathematical ideas behind this note are, essentially, two tricks: Scott's famous trick \cite{Dana}, and an application of Gabriel-Popescu Theorem to show that limits in any Grothendieck theory can be computed uniformly just like in \textbf{SET}. This later result (Theorem \ref{only-theorem}) is trickier than it seems and has some independent interest.

As a last foundational remark before we start, we point out that since we rely heavily on Scott's trick, our approach also makes a point against a commonly held view (see, for instance, \cite{Kunen}) that the axiom of regularity is somewhat superfluous for the actual mathematical practice.

\section{Derived functors in ZFC}

Now we properly start. For the sake of completeness, we recall the notion of a Grothendieck category:

\begin{definition}\cite{Grothendieck}
An abelian category $\mathfrak{Gr}$ is called a Grothendieck category if it satisfies the axiom \textbf{AB5} (cf. \cite{Grothendieck}) that is, the inductive limit functor induced by a directed set is exact; and if it has a family of generators $\{ U_i \}$. 
\end{definition}

Examples of Grothendieck categories abound: we can mention the category $R-\operatorname{Mod}$ for an unital ring $R$, with $\mathfrak{Ab}$ corresponding to $\mathbb{Z}$-modules. Another example is, given a ringed space $(X, \mathcal{O}_X)$, the cateogry of sheaves of $\mathcal{O}_X$-modules.

To our approach to derived functors, we will need to know that for every Grothendieck category, limits of diagrams can be computed as in \textbf{SET} (e.g., \cite{Stacks}):

\begin{proposition}
Let $\mathcal{D}: \mathcal{I} \rightarrow $ \textbf{SET} be a diagram. Then $\operatorname{Lim}_\mathcal{I} \mathcal{D}$ can be constructed explicitly as

\[\{ \langle \mathcal{D}(i)) \rangle_{i \in \mathcal{I}}|(\forall \phi) \phi: i \rightarrow i', \mathcal{D}(\phi)(\mathcal{D}(i))=\mathcal{D}(i') \}. \]
\end{proposition}

\medskip

 We recall that a concrete category $\mathcal{C}$ is one which admits a fully faithfull embedding in the category of sets, and that a category is locally small if its $\mathsf{Hom}$ sets are not proper classes. Grothendieck categories are locally small (\cite[5.35]{BD}), and concrete, by Gabriel-Popescu Theorem:

 \begin{theorem}\cite{GP}
Let $\mathfrak{Gr}$ be a Grothendieck category. Then there is a unital ring $R$ and functors $G: \mathfrak{Gr} \rightarrow R-\operatorname{Mod}$ and $F: R-\operatorname{Mod} \rightarrow \mathfrak{Gr} $ such that $G$ is fully faithfull, $F$ is its left adjoint and exact.
 \end{theorem}

 An important consequence of this is that the functor $G$ is monadic \cite{Joy}, and regarding these:

 \begin{proposition}\label{creates}
     A monadic functor creates all limits in its codomain.
 \end{proposition}
\begin{proof} \cite[Exercise IV.2.2 (p. 138)]{MacLane}.
\end{proof}

We need  another definition and Proposition.

\begin{definition}
    Let $\mathcal{B}$ be a subcategory of a category $\mathcal{C}$. Then it is called replete if for every $x \in \operatorname{Obj} \, \mathcal{B}, y \in \operatorname{Obj} \, \mathcal{C}$, if $f \in \operatorname{Hom}_\mathcal{C}(x,y)$ is an isomorphism, then $y$ is in $\mathcal{B}$ and $f \in \operatorname{Hom}_\mathcal{B}(x,y)$.
\end{definition}

Since being replete is a closure condition, by the repletion $\operatorname{repl}(\mathcal{B})$ of a subcategory $\mathcal{B}$ of a category $\mathcal{C}$ we mean the smallest replete subcategory of $\mathcal{C}$ containing it.

\begin{proposition}\label{equivalent}
If $\mathcal{B}$ is a full subcategory of $\mathcal{C}$, then $\operatorname{repl}(\mathcal{B})$ and $\mathcal{B}$ are equivalent.
\end{proposition}
\begin{proof}
Every morphism $\mathfrak{f}$ in $\operatorname{repl}(\mathcal{B})$ factors as  $g \circ f \circ h^{-1}$, where $f$ is a morphism of $\mathcal{B}$ and $h,g$ isomorphisms in $\mathcal{C}$. If, say, the domain and codomain of $\mathfrak{f}$ are in $\mathcal{B}$, then $g$ and $h$, and hence also $\mathfrak{f}$ must actually belong to $\mathcal{B}$. Hence the inclusion $\mathcal{B} \into \operatorname{repl}(\mathcal{B})$ is fully faithfull and essentially surjective, and we are done.
\end{proof}

With this preparation, finally we can show

\begin{theorem}\label{only-theorem}
    Limits in a Grothendieck category can be constructed exactly like in \textbf{SET}.
\end{theorem}
\begin{proof}
    First we remark that if our category is $R$-$\operatorname{Mod}$ for an unital ring, this is clear. The technical reason is that we have that the forgetful functor from this category to \textbf{SET} is amnestic (cf. \cite[Chap I, Definition 3.27]{Joy}). If $\mathfrak{Gr}$ is a Grothendieck category and $F:\mathfrak{Gr} \rightarrow R-\operatorname{Mod}$ is the monadic functor from Gabriel-Popescu Theorem, then by Proposition \ref{creates} everything would be finished if we had a guarantee that the limit of a diagram in $\mathfrak{Gr}$, constructed as usual in $R-\operatorname{Mod}$, \emph{still} lies in $\mathfrak{Gr}$. However, by the previous Proposition, we can replace $\mathfrak{Gr}$ by its repletion. In this case the limit has the desired property.
\end{proof}

We now can define the left derived functors $L_n(F)$ of a right exact functor $F: \mathcal{C} \rightarrow \mathfrak{Gr}$ from an abelian category to a Grothendieck category.

The following is the sole new definition of this note\footnote{And it is here that the Axiom of Regularity is crucial.}.

\begin{definition}
    Let $M$ be an object in a concrete locally small abelian category $\mathcal{C}$, with a functor $\iota: \mathcal{C} \rightarrow \mathsf{SET}$ which is a fully faithful embedding, and $(\mathsf{P}, \epsilon)$ a projective resolution of it. We write $\operatorname{rk}(\mathsf{P})$ for the the ordinal $\operatorname{sup}\{ \beta \in \mathsf{ORD}| (\exists n \in \mathbb{N}) \text{ such that} \operatorname{rk}(\iota(P_n)=\beta \}$.
\end{definition}

By Scott's trick (cf. \cite[p. 65]{Jech}), using only the axiom schema of replacement (the axiom of choice is \emph{not} needed in this step), we have a proper class function $\bar{F}: \mathcal{C} \rightarrow \mathsf{V}$ that sends each object $M$ to the \emph{set} $\bar{F}(M)$ of projective resolutions of minimal rank.

Now, given $\mathcal{I}=\bar{F}(M)$, we have that if $(\mathsf{P}^i, \epsilon_i), i=1,2,3$ belongs to $\mathcal{I}$, by \cite[Proposition 5.1]{HS} there are \emph{canonical} isomorphism $\phi_{ij}:H^n(F(\mathsf{P}^i)) \rightarrow  H^n(F(\mathsf{P}^j))$. We have that $\phi_{ii}=\operatorname{id}$, $\phi_{jk}\phi_{ij}=\phi_{ik}$, and $\phi_{ij}=\phi_{ji}^{-1}$.


Our desired object is now simple to describe: we can turn $\mathcal{I}$ into a small category such that, given two objects $i,j \in \mathcal{I}$, there is only one isomorphism between them, and we want to define $L_n(F)(M)=\operatorname{Lim}_\mathcal{I}H^n(F(\mathsf{P}_i))$ as the limit indexed by $\mathcal{I}$. There is, however, a small techinical detail: usually, limits are not unique - they are only unique up to isomophism. So, à priori, we would need the axiom of global choice to choose $L_n(F)$ uniquely once and for all. However, in our case, thanks to Theorem \ref{only-theorem}, we can deal with limits uniformly: they have the same construction as in \textbf{SET}.




Clearly, given any resolution $(\mathsf{Q}, \epsilon)$ of $M$, $H^n(F(\mathsf{Q}))$ is isomorphic to $L_n(F)(M)$. 

Having defined $L_n(F)$ for objects, it is simple to define it for morphisms. If $f:M \rightarrow N$ is a morphism between two objects in $\mathcal{C}$, and $(\mathsf{P}^1, \epsilon_1)$ and $(\mathsf{P}^1, \epsilon_2)$ are projective resolutions of $M$ and $N$ respectively, by \cite[Theorem 4.1]{HS} the map $f$ can be extended uniquely up to homotopical equivalence to a chain map $\bar{f}$ between $\mathsf{P}^1$ and $\mathsf{P}^2$. Let $\mathbb{F}$ be the collection of all chain maps $\bar{f}^*$ that extends $f$. Since we have a countable number of objects in a chain complex, and our category is locally small, $\mathbb{F}$ is a honest set, not a proper class. We define on $\mathbb{F}$ an equivalence relation $\mathsf{f} \sim \mathsf{g}$ if and only if they induce the same morphism  between $H^n(F(\mathsf{P}^1))$ and $H^n(F(\mathsf{P}^2))$. By \cite[Corollary 3.5]{HS}, $\sim$ has only one equivalence class, that is, the whole $\mathbb{F}$. We define this unique morphism between $H^n(F(\mathsf{P}^1))$ and $H^n(F(\mathsf{P}^2))$ as the \emph{same} map on cohomology induced by \emph{any} element of $\mathbb{F}$\footnote{the introduction of $\mathbb{F}$ and the equivalence relation $\sim$ was done to avoid the use of any form of the axiom of choice}, which we will denote by $\theta_n(M,N, \mathsf{P}^1, \mathsf{P}^2)$. If $(\mathsf{Q}^1, \epsilon_1')$ and $(\mathsf{Q}^2, \epsilon_2')$ are two other projective resolutions of $M$ and $N$ respectively, denoting the canonical isomorphism between $H^n(F(\mathsf{P}^1))$ and $H^n(F(\mathsf{Q}^1))$ by $\phi_n(\mathsf{P}^1, \mathsf{Q}^1)$ and the one between $H^n(F(\mathsf{P}^2))$ and $H^n(F(\mathsf{Q}^2))$ by $\phi_n(\mathsf{P}^2, \mathsf{Q}^2)$, we have (\cite[p. 132]{HS}) that

\[(1) \, \,  \theta_n(M,N,\mathsf{P}^1,\mathsf{P}^1)=\phi_n(\mathsf{P}^1,\mathsf{Q}^1) \theta_n(M,N,\mathsf{Q}^1,\mathsf{Q}^1) \phi_n(\mathsf{P}^2,\mathsf{Q}^2)^{-1} \]

\medskip

Let's write $\mathcal{I}=\bar{F}(M)$ and $\mathcal{I}^*=\bar{F}(N)$ for the small categories used in the construction of the derived functors on the objects. Let $i \in \mathcal{I}$ and consider $H^n(F(\mathsf{P}^i))$. For each $i^* \in \mathcal{I}^*$ consider the corresponding cohomology group $H^n(F(\mathsf{Q}^{i^*}))$. Consider the maps $\rho_{i}:H^n(F(\mathsf{P}^i)) \rightarrow \prod_{i^* \in I^*}H^n(F(\mathsf{Q}^{i^*}))$ given by $\langle \theta_n(M,N,\mathsf{P}^i,\mathsf{Q}^{i^*}) \rangle_{i^* \in I^*}$. Because of equation (1), this map induces a map $\bar{\rho_i}:H^n(F(\mathsf{P}^i)) \rightarrow \operatorname{Lim}_{\mathcal{I}^*} H^n(F(\mathsf{Q}^{i^*}))$. Since, again by equation (1), $\bar{\rho_{i}} \phi_{ij}=\bar{\rho_{j}}$, we have finally a map $L_n(F)(f):L_n(F)(M) \rightarrow L_n(F)(N)$.

It now remains to show that if we have $M \xrightarrow{f} N \xrightarrow{g} O$, $L_n(F)(g \circ f)=L_n(F)(g) \circ L_n(F)(f)$. If $(\mathsf{P}^1,\epsilon_1), (\mathsf{P}^2, \epsilon_2), (\mathsf{P}^3,\epsilon_3)$ are projective resolutions of $M, N, O$ respectively, if $\bar{f}$ is chain map extending $f$, $\bar{g}$ is a chain map extending $g$, it is clear that $\bar{g} \circ \bar{f}$ is a chain map between $(\mathsf{P}^1,\epsilon_1)$ and $(\mathsf{P}^3,\epsilon_3)$ extending $g \circ f$. Again, by \cite[Corollary 3.5]{HS}, we have that

\[ \theta_n(N, O, \mathsf{P}^2, \mathsf{P}^3) \circ \theta_n(M,N, \mathsf{P}^1, \mathsf{P}^2) = \theta_n(M, O, \mathsf{P}^1, \mathsf{P}^3), \]

and same as before, using the compatibility relations of equation (1), we get $L_n(F)(g \circ f)=L_n(F)(g) \circ L_n(f)$.

Also, it is clear by the same \cite[Corollary 3.5]{HS} that $L_n(F)(\operatorname{Id}_M)=\operatorname{Id}$. So we finally can say that we have a functor.

That they are $\delta$-functors in the sense of Grothendieck \cite{Grothendieck} can be carried out with by now clear modifications from the expositions in standard textbooks on homological algebra. And, as it is clear, every theorem of this theory can be carried out in \textbf{ZFC}.

\section{Metamathematical consideradions}

Finally, some metamathematical remarks will be made now about the choice of our framework. They are aimed specially at the general mathematician, as they are well known in the logic community. A naive alternative to what we did would be: work in a set theory that has proper classes like \textbf{NBG} with the axiom of global choice included, and among every proper class of isomorphisms classes of cohomology groups, select a representative. \textbf{NBG} plus the axiom of global choice is conservative over \textbf{ZFC} (\cite[p. 134]{Foundations}), but as we mentioned before, localizing a category can ``explode" to the next Universe, and it would be very cumbersome to deal with this fact with set theory with proper classes, even if a stronger one, like \textbf{MK}, was considered (cf. \cite{Shulman}). Also, our approach naturally identifies each cohomology group in a very natural way, which avoids such arbitrary choices. The other usual alternative is to assume the existence of enough Grothendieck universes. However, this is a somewhat strong metamathematical commitment, for even the theory \textbf{ZFC+Cons(ZFC+U)}, where \textbf{Cons(ZFC+U)} is the sentence\footnote{In \textbf{ZFC}.} that says 'the theory \textbf{ZFC} adjoined with an axiom postulating the existence of at least one Grothendieck universe is consistent' demands a stronger consistency assumption than just \textbf{ZFC} alone (cf. \cite[p. 145]{Kunen}).

Of course there are other approaches to the set-theoretical foundations of category theory (see, e.g., \cite{Shulman}); we mention, in particular, reflection principles (cf. \cite{Kunen} \cite{Jech}). But we think that our approach has the merit of being as simple as possible and remaining in plain \textbf{ZFC}.


The author would like to acknowledge the important discussions with H. L. Mariano and R. Freire, who also read preliminary versions of this note, the the kind and useful feedback from D. Roberts, who read the first arXiv version of this preprint and A. L. Tenório, which whom the author discussed the final form of this manuscript.


\end{document}